\documentclass{amsart}

\usepackage{mathrsfs}
\usepackage[centertags]{amsmath}
\usepackage{amsfonts}
\usepackage{amsthm}
\usepackage{amssymb}
\usepackage{color}
\usepackage[colorinlistoftodos,prependcaption]{todonotes}
\usepackage[colorlinks,linkcolor={blue},citecolor={blue},urlcolor={red}]{hyperref}
\usepackage{enumerate}

\let\emptyset \undefined
\let\ge       \undefined
\let\le       \undefined
\newsymbol\le          1336  \let\leq\le
\newsymbol\ge          133E  \let\geq\ge
\newsymbol\emptyset    203F
\newsymbol\notle       230A
\newsymbol\notge       230B

\allowdisplaybreaks

\newtheorem{theorem}{Theorem}[section]
\theoremstyle{remark}
\newtheorem{remark}[theorem]{Remark}
\newtheorem{example}[theorem]{Example}
\theoremstyle{plain}
\newtheorem{corollary}[theorem]{Corollary}
\newtheorem{lemma}[theorem]{Lemma}
\newtheorem{proposition}[theorem]{Proposition}
\newtheorem{definition}[theorem]{Definition}
\newtheorem{hypothesis}[theorem]{Hypothesis}

\newtheorem{question}[theorem]{Question}

\numberwithin{equation}{section}

\def\N{{\mathbb N}}

\def\R{{\mathbb R}}

\newcommand{\Om}{\Omega}
\newcommand{\om}{\omega}
\newcommand{\F}{\mathscr{F}}
\newcommand{\G}{\mathscr{G}}
\renewcommand{\P}{\mathbb{P}}
\newcommand{\E}{\mathbb{E}}

\newcommand{\la}{\lambda}
\newcommand{\eps}{\varepsilon}
\newcommand{\calL}{\mathscr{L}}
\newcommand{\n}{\|}
\newcommand{\one}{{\bf 1}}

\newcommand{\lb}{\langle}
\newcommand{\rb}{\rangle}
\newcommand{\limn}{\lim_{n\to\infty}}

\newcommand{\ud}{\,{\rm d}}

\renewcommand{\H}{\mathscr{H}}
\newcommand{\Dom}{\mathsf{D}}

\newcommand{\bF}{\mathscr{P}}

\allowdisplaybreaks

\begin{document}

\title[Stochastic convolutions in 2-smooth Banach spaces]
{Maximal estimates for stochastic convolutions in 2-smooth Banach spaces and applications to stochastic evolution equations}

\author{Jan van Neerven, Mark Veraar}

\address{Delft Institute of Applied Mathematics\\
Delft University of Technology\\
P.O. Box 5031, 2600 GA Delft\\
The Netherlands}
\email{J.M.A.M.vanNeerven/M.C.Veraar@TUDelft.nl}

\thanks{The second named author is supported by VIDI subsidy 639.032.427
of the Netherlands Organisation for Scientific Research (NWO)}


\begin{abstract}
This paper presents a survey of maximal inequalities for stochastic convolutions in $2$-smooth Banach spaces and their applications to stochastic evolution equations.
\end{abstract}

\maketitle

%\tableofcontents

\section{Introduction}

This paper presents an overview of maximal inequalities for Banach space-valued stochastic processes $(u_t)_{t\in [0,T]}$ of the form
\begin{align}\label{eq:ut} u_t = \int_0^t S(t,s)g_s \ud W_s,\qquad t\in [0,T],
\end{align}
where $(S(t,s))_{0\le s\le t\le T}$ is a strongly continuous evolution family acting on a Banach space $X$,
$(W_t)_{t\in [0,T]}$ is a (cylindrical) Brownian motion defined on a probability space $\Om$, and $(g_t)_{t\in [0,T]}$ is a stochastic process taking values in $X$ (in the case of a Brownian motion $(W_t)_{t\in [0,T]}$)
or in a space of operators acting from $H$ to $X$ (in the case of a cylindrical Brownian motion whose covariance is given by the inner product of a Hilbert space $H$), defined on the same probability space $\Om$. The stochastic integral in \eqref{eq:ut} is the Banach space-valued extension of the classical It\^o stochastic integral. In the important special case $S(t,s) = S(t-s)$ arising from a one-parameter semigroup of operators $(S(t))_{t\ge 0}$, the stochastic integral \eqref{eq:ut} takes the form of a stochastic convolution. This justifies our slight abuse of terminology to also refer to \eqref{eq:ut} as a stochastic convolution. In addition to reviewing the literature on this topic, some new contributions are included as well.

Under a {\em maximal inequality} for $(u_t)_{t\in [0,T]}$ we understand a bound on the random variable
$$ u^\star(\om) = \sup_{t\in [0,T]} \n u_t(\om)\n, \qquad \om\in\Om.$$
Maximal inequalities are important in the theory of stochastic evolution equations, where the mild solution of the time-dependent inhomogeneous stochastic evolution equation
\begin{align*} \begin{cases}
\ud u_t            &\!\!\!\!\! = A(t) u_t\ud t + g_t\ud W_t,  \qquad t\in [0,T],\\
\phantom{\ud} u_0  &\!\!\!\!\! = 0  \end{cases}
\end{align*}
is of the form \eqref{eq:ut} provided one assumes that the operator family $(A(t))_{t\in [0,T]}$ generates the evolution family $(S(t,s))_{0\le s\le t\le T}$ in a suitable sense.
The availability of a maximal inequality in this setting typically implies that the solution process $(u_t)_{t\in [0,T]}$ has a continuous version.

In the present paper we limit ourselves to maximal estimates of Burkholder type, where $u^\star$ is estimated in terms of a square function norm analogous to the one occurring in the classical Burkholder maximal inequality for continuous time martingales.
Different techniques to obtain pathwise continuous solutions, such as developed in
\cite{BGIPPZ, IMMTZ}, will not be discussed here.

At present, two theories of It\^o stochastic calculus in Banach spaces are available: for $2$-smooth Banach spaces \cite{Nei} and for UMD Banach spaces \cite{NVW07a, NVW07b}. Both approaches are surveyed in \cite{NVW13}.
Each of the two approaches has its advantages and disadvantages. The construction of the stochastic integral in $2$-smooth Banach spaces is fairly elementary and its use in the theory of stochastic evolution equations is straightforward, but its applicability covers only half of the $L^p$-scale (namely the exponents $2\le p<\infty$) \cite{Brz95, Brz97}. It replaces the basic It\^o isometry with a one-sided estimate which necessarily entails some loss in precision. This manifests itself in questions relating to maximal regularity, which cannot be fully treated with this theory.
The stochastic integral in UMD Banach spaces covers the full reflexive $L^p$-scale (exponents $1< p<\infty$) and leads to two-sided estimates for the stochastic integrals, but due to the more subtle form of the expressions involved it requires some constraints on the properties of the stochastic processes under consideration. In practice this entails that the theory can be applied effectively to evolution equations in the parabolic setting only, but in that setting a full-fledged maximal regularity theory is available \cite{AgreVernon, NVW12a,NVW12b,  PorVer}. The $2$-smooth theory is applicable beyond the parabolic setting and covers the case of arbitrary $C_0$-evolution families.

In order to keep this paper at a reasonable length we will exclusively deal with maximal estimates in the $2$-smooth setting.
Maximal estimates for stochastic convolutions in $2$-smooth Banach spaces are useful in applications to stochastic partial differential equations. Typically one takes $X$ to be $L^p$, the Bessel potential space $H^{s,p}$, or the Besov space $B^{s}_{p,q}$; as will be explained in Example \ref{ex:2smooth} these spaces are $2$-smooth if $2\le p,q< \infty$.
Maximal estimates in the setting of UMD spaces are covered in \cite{VerWei} and the follow-up works \cite{NVW12a,NVW12b}.
Maximal inequalities for stochastic convolutions in $2$-smooth Banach spaces with respect to other noise processes than (cylindrical) Brownian motions, such as Poissonian noise, are discussed in \cite{HauSei2, ZBH, ZhuBrzLiu} and the references therein; see also survey \cite{MarRoc14} for the Hilbertian case.

An important motivation to study stochastic partial differential equations in the setting of $2$-smooth Banach spaces comes from the fact that estimates in Sobolev and Besov spaces with high integrability exponent $p$ (which are $2$-smooth) can be combined with Sobolev embedding results to obtain obtain further integrability and regularity properties of solutions. This plays a key role in many papers (see \cite{AgreVernon,Brz97} and references therein).

\section{Preliminaries}\label{sec:preliminaries}

We assume familiarity with the basic notions of probability theory and stochastic analysis. This preliminary section fixes notation  following  the references \cite{NVW13,HNVW16, HNVW17} where unexplained terminology can be found.
All random variables and stochastic processes are assumed to be defined on a probability space $(\Om,\F,\P)$ which we fix once and for all. We work over the real scalar field.

\subsection{Stochastic preliminaries}\label{subsec:stoch}

When $X$ is a Banach space, an $X$-valued {\em random variable} is a strongly measurable
function $\phi:\Om\to X$, i.e., a function that can be approximated $\P$-almost surely by a sequence of $\F$-measurable simple functions with values in $X$. The adjective `$X$-valued' will be usually omitted; depending on the context, random variables can be real- or vector-valued.
The expected value of an integrable random variable $\phi$ is denoted by $\E \phi = \int_\Om \phi \ud \P$. For $0< p\le \infty$
we denote by $L^p(\Om;X)$ the (quasi-) Banach space of strongly measurable functions $\phi:\Om\to X$ such that $\E \n \phi\n^p<\infty$, with the usual adjustment for $p=\infty$, and by $L^0(\Om;X)$ the space of all strongly measurable functions $\phi:\Om\to X$ endowed with the metric topology induced by convergence in measure. In dealing with elements of these spaces it is always understood that we identify random variables that equal almost surely. When $0\le p\le \infty$ and $\G$ is a sub-$\sigma$-algebra of $\F$, we denote by $L^p(\Om,\G;X)$ the closed subspace of $L^p(\Om;X)$ of all elements that are strongly measurable as random variables defined on $(\Om,\G,\P|_\G)$. The conditional expectation of a random variable $\phi$ given $\G$ is denoted by $\E_\G(\phi)$ or $\E(\phi|\G)$.

A {\em filtration} is a family $(\F_t)_{t\in [0,T]}$ of sub-$\sigma$-algebras of $\F$ such that $\F_s\subseteq \F_t$ whenever $s\le t$. A {\em process} is a family of $X$-valued random variables $(\phi_t)_{t\in [0,T]}$. It is called {\em adapted} if for every $t\in [0,T]$ the random variable $\phi_t$ is strongly measurable as a random variable on $(\Om,\F_t,\P|_{\F_t})$. A process $\phi$ is called a {\em martingale} if $\E(\phi_t|\F_s) = \phi_s$ almost surely whenever $s\le t$.
Discrete filtrations and martingales are defined similarly, replacing the index set $[0,T]$ by a finite set $\{0,1,\dots,N\}$.

The {\em progressive $\sigma$-algebra} on $[0,T]\times\Omega$ is the $\sigma$-algebra $\mathscr{P}$ generated by sets of the form $B\times A$ with $B\in \mathcal{B}([0,t])$ and $A\in \F_t$, where $t$ ranges over $[0,T]$. A process $\phi$ is said to be {\em progressively measurable} if it is strongly measurable with respect to $\mathscr{P}$.
Two processes $\phi,\psi$ are called {\em versions} of each other if for every $t\in [0,T]$ we have $\phi_t = \psi_t$ almost surely;
the exceptional set is allowed to depend on $t$. A process $\phi$ is said to have a {\em continuous version} if it has a pathwise continuous version $\psi$, i.e., a version such that for all $\om\in \Om$ the path $t\mapsto \psi_t(\om)$ is continuous.

Next we extend the notion of a Hilbert--Schmidt operator to the Banach space setting. The reader is referred to  \cite[Chapter 9]{HNVW17} and \cite{Nee10} for systematic treatments. Let $\H$ be a Hilbert space and $X$ be a Banach space.
The space of finite rank operators from $H$ into $X$ is denoted by $\H\otimes X$.
For a finite rank operator $R\in \H\otimes X$, say $R = \sum_{n=1}^N h_n\otimes x_n$ with $(h_n)_{n=1}^N$ orthonormal in $\H$ and
$(x_n)_{n=1}^N$ a sequence in $X$ (we can always represent $R$ in this way by a Gram--Schmidt orthogonalisation argument), we define
\begin{equation*} \n R\n_{\gamma(\H,X)}^2 = \E \Big\n \sum_{n=1}^N \gamma_n x_n\Big\n^2,
\end{equation*}
where $(\gamma_n)_{n=1}^N$ is a sequence of independent standard Gaussian random variables. The norm  $\n \cdot\n_{\gamma(\H,X)}$ is well defined, and the completion of $\H\otimes X$ with respect to this norm is denoted by $\gamma(\H,X)$. The natural inclusion mapping $\H\otimes X\subseteq \calL(\H,X)$ extends to an injective and contractive inclusion mapping $\gamma(\H,X)\subseteq \calL(\H,X)$.
A linear operator in $\calL(\H, X)$ is said to be {\em $\gamma$-radonifying} if it belongs to $\gamma(\H,X)$.
For Hilbert spaces $X$, the identity mapping on $\H\otimes X$ extends to an isometrical isomorphism $$\gamma(\H,X) \simeq \calL_2(\H,X),$$ where $\calL_2(\H,X)$ is the space of Hilbert--Schmidt operators from $\H$ to $X$.
Another important instance where an explicit identification is available is the case $X = L^p(S;Y)$ with $(S,\mathscr{A},\mu)$ a measure space, $1\le p<\infty$, and $Y$ a Banach space; the mapping $h\otimes (f\otimes y) \mapsto f\otimes (h\otimes y)$ sets up an isomorphism of Banach spaces $$ \gamma(\H, L^p(S;Y)) \simeq L^p(S;\gamma(\H,Y)).$$

\subsection{$2$-Smooth Banach spaces}\label{subsec:2smooth}

A Banach space $X$ is called {\em $(p,D)$-smooth}, where $p\in [1,2]$ and $D\ge 0$ is a constant, if for all $x,y\in X$ we have
\begin{align}\label{eq:K}
	 \n  x+y\n^p+\n  x-y\n^p\leq 2\n  x\n^p+2D^p\n  y\n^p.
\end{align}
A Banach space is called {\em $p$-smooth} if it is $(p,D)$-smooth for some $D\ge 0$.
The case $x=0$ demonstrates that the constant in \eqref{eq:K} necessarily satisfies $D\ge 1$. For $p=2$ the defining condition is a generalised parallelogram identity.

\begin{example}\label{ex:2smooth}
Here are some examples of $2$-smooth Banach spaces:
\begin{itemize}
 \item Every Hilbert space  is $2$-smooth (with $D=1$, by the parallelogram identity).
 \item The space $L^p(\mu)$ is $2$-smooth if and only if $2\le p<\infty$ (and in that case we may take $D=\sqrt{p-1}$, see \cite[Proposition 2.1]{Pin}).
 More generally, the space $L^p(\mu;X)$ is $2$-smooth if and only if $X$ is $2$-smooth and $2\le p<\infty$
(in that case, if $X$ is $(2,D)$-smooth, then $L^p(\mu;X)$ is $(2,D\sqrt{p-1})$-smooth \cite{NV20}; for an earlier result in this direction see \cite{Fig}).
 \item For any $s\in \R$, the Bessel potential space $H^{s,p}(\R^d)$ is $2$-smooth if and only if $2\le p<\infty$ (and in that case we may take $D=\sqrt{p-1}$). Indeed, this space is isometrically isomorphic to $L^p(\R^d)$, the isometry being given by the Fourier multiplier $(1+|\xi|^2)^{s/2}$.
  \item For any $s\in \R$, the Besov space $B^{s}_{p,q}(\R^d)$ (equipped with its Littlewood--Paley norm) is $2$-smooth if and only if  $2\le p,q<\infty$ (and in that case we may take $D=\sqrt{(p-1)(q-1)}$) Indeed, with this norm, Littlewood--Paley theory identifies $B^{s}_{p,q}(\R^d)$ isometrically with a closed subspace of $\ell^p(L^q(\R^d))$. Analogous considerations apply to the Triebel--Lizorkin spaces $F^{s}_{p,q}(\R^d)$, and for both scales the results extend to more general open domains $\mathscr{O}\subseteq \R^d$ (with the same constant $D$ if one uses the quotient norm of \cite[Section 4.2.1]{Tri}).
  \item For any $k\in\N$, the Sobolev space $W^{k,p}(\mathscr{O})$ is $2$-smooth if and only if $2\le p<\infty$ (and in that case we may take $D=\sqrt{p-1}$ if we use the norm $\n f\n_{W^{k,p}(\mathscr{O})}^p = \sum_{|\alpha|\le k} \n \partial^\alpha f\n_p^p$).
  \item The Schatten trace ideal $C^p$ is $2$-smooth if and only if $2\le p<\infty$ (and in that case we may take $D=\sqrt{p-1}$, see \cite{BCL}).
\end{itemize}
\end{example}

A Banach space $X$ is said to have \emph{martingale type} $p\in [1,2]$\index{martingale type $p$} if there exists a constant $C\ge 0$ such that
\begin{equation}\label{eq:def-marttype}
  \E \n f_N\n {p}
  \leq C^p \Bigl(\E \n f_0\n ^p+\sum_{n=1}^N\n f_n - f_{n-1}\n^p\Bigl)
\end{equation}
for all $X$-valued $L^p$-martingales $(f_n)_{n=0}^N$. The case $N=0$ demonstrates that the constant in \eqref{eq:def-marttype} necessarily satisfies $C\ge 1$.
It is a fundamental result due to Pisier \cite{Pis} (see also \cite{HNVW4, Wen, Woy}) that, for any $p\in [1,2]$, up to equivalence of norms, a Banach space is $p$-smooth if and only if it has martingale type $p$.
The advantage of $p$-smoothness over martingale type $p$ is that the former is an isometric condition, whereas the latter is  isomorphic. We will encounter various maximal inequalities for semigroups or evolution families of contractions acting on $2$-smooth Banach spaces. Such results cannot be expected to have a counterpart in martingale type $2$-spaces, unless they hold more generally for uniformly bounded $C_0$-semigroups, the point being that contractivity is typically not preserved under passing to equivalent norms.

\section{Maximal inequalities for indefinite stochastic integrals}\label{sec:max-indef}

Let $\H$ be a Hilbert space with inner product $(\cdot|\cdot)$.
An {\em $\H$-isonormal process} is a
mapping $W: \H\to L^2(\Om)$ with the following two properties:
\begin{enumerate}
\item[\rm(i)] for all $h\in \H$ the random variable $Wh$ is Gaussian;
\item[\rm(ii)] for all $h_1,h_2\in \H$ we have
$\E (Wh_1 \cdot Wh_2) = (h_1|h_2)$.
\end{enumerate}
For $h=0$ we interpret $W0$ as the Dirac measure concentrated at $0$. From (ii) it easily follows that $\H$-isonormal processes are linear, and this in turn implies that
for all $h_1,\dots,h_N\in \H$ the $\R^N$-valued random variable
$(Wh_1,\dots, Wh_N)$ is jointly Gaussian, i.e.,
$(Wh)_{h\in H}$ is a Gaussian process; see \cite{Nua} for the details.

If $W$ is an $L^2(0,T)$-isonormal process, the process $(W\one_{(0,t)})_{t\in [0,T]}$ is a standard Brownian motion. This prompts us to define,
for a Hilbert space $H$, a {\em cylindrical $H$-Brownian motion} as an $L^2(0,T;H)$-isonormal process. In what follows the Hilbert space $H$ will be considered to be fixed and we will consider a fixed cylindrical $H$-Brownian motion $W$.
Following standard usage
in the literature we will write
$$ W_th:= W(\one_{(0,t)}\otimes h),\qquad t\in [0,T], \ h\in H.$$
For each $h\in H$,
$ (W_th)_{t\in [0,T]}$ is a Brownian motion, which is standard if and only if $h$ has norm one; two such Brownian motions
corresponding to $h_1,h_2\in H$ are independent if and only if $h_1$ and $h_2$ are orthogonal.
A cylindrical $H$-Brownian motion $W$ is said to be {\em adapted} to a given
filtration $(\F_t)_{t\in [0,T]}$ on $(\Om,\F,\P)$ if $W_th\in L^2(\Om,\F_t)$ for all $t\in [0,T]$ and $h\in H$.
In what follows, we will always assume that a filtration has been fixed and that $W$ is adapted to it.

A stochastic process $\Phi :[0,T]\times \Om\to \gamma(H,X)$ is called an {\em adapted finite rank step process} if there exist $0=s_0<s_1<\ldots<s_n=T$, random variables $\xi_{ij}\in L^\infty(\Omega,\F_{s_{j-1}})\otimes X$ (the subspace of $L^\infty(\Omega;X)$ of strongly $\F_{s_{j-1}}$-measurable random variables taking values in a finite-dimensional subspace of $X$)
for $i=1, \ldots, m$ and $j=1, \ldots, n$, and an orthonormal system $h_1,\dots, h_m$ in $H$ such that
\begin{equation*}%\label{eq:simple}
\Phi  = \sum_{j=1}^{n} \one_{(s_{j-1},s_{j}]} \otimes \sum_{i=1}^m h_i\otimes \xi_{ij}.
\end{equation*}
The {\em stochastic integral process} associated with $\Phi$ is then defined by
\[\int_0^t \Phi_s \ud W_s := \sum_{j=1}^n  \sum_{i=1}^m  (W_{s_{j}\wedge t}-W_{s_{j-1}\wedge t})h_i \otimes \xi_{ij}, \qquad t\in [0,T].\]
Since $s\mapsto W_sh$, being a Brownian motion, has a continuous version, it follows that the process $t\mapsto \int_0^t\Phi_s \ud W_s$ has a continuous version. Such versions will always be used in the sequel.

The following elementary upper bound for the stochastic integral of $X$-valued elementary adapted processes with respect to the cylindrical Brownian motion $W$, due to Neidhardt \cite{Nei}, extends the It\^o isometry to $2$-smooth Banach spaces. It is important to note that the proposition only provides an upper bound. It can be shown that this upper bound is an equivalence of norms if and only if $X$ is isomorphic to a Hilbert space \cite{NVW13}. Indeed it is this one-sidedness of the bound which constitutes the main limitation of the It\^o stochastic integral in $2$-smooth Banach spaces compared to its competitor for UMD Banach spaces.

\begin{proposition}[Neidhardt]\label{prop:Neid}
 Let $X$ be a $(2,D)$-smooth Banach space. Then, for all adapted finite rank step processes
$\Phi:[0,T]\times \Omega\to \gamma(H,X)$,
$$ \E \Big\n \int_0^T \Phi_t \ud W_t\Big\n^2 \le D^2 \|\Phi\|_{L^2(\Omega;L^2(0,T;\gamma(H,X)))}^2.
$$
\end{proposition}

Since the adapted finite rank step processes are dense in the closed subspace $L_{\bF}^2(\Om;L^2(0,T;\gamma(H,X)))$ consisting
of all progressively measurable
processes in $L^2(\Om;L^2(0,T;\gamma(H,X)))$, the estimate of Proposition \ref{prop:Neid} permits the extension of the stochastic integral to processes $\Phi\in L_{\bF}^2(\Omega;L^2(0,T;\gamma(H,X)))$.
By Doob's maximal inequality the resulting stochastic integral process $t\mapsto \int_0^t \Phi_s\ud W_s$ has a continuous version which satisfies the maximal estimate
\begin{align}\label{eq:maxp2}
\E \sup_{t\in [0,T]}\Big\n \int_0^t \Phi_s \ud W_s\Big\n^2 \le 4D^2 \|\Phi\|_{L^2(\Omega;L^2(0,T;\gamma(H,X)))}^2.
\end{align}
By a standard localization argument the mapping $\Phi \mapsto \int_0^\cdot \Phi_s\ud W_s$ can be extended to a continuous mapping from $L_{\bF}^0(\Omega;L^2(0,T;\gamma(H,X)))$ into $L^0(\Omega;C[0,T];X))$. Here, and in other instances below, the subscript $\bF$ designates the closed subspace of all progressively measurable process in a given space of processes.

In the scalar-valued setting it is a classical result of Burkholder, with later refinements by Davis and Gundy,
that the maximal inequality \eqref{eq:maxp2} admits an extension with $L^2$-norms over $\Omega$ replaced by $L^p$-norms
with constants of order $O(\sqrt{p})$ as $p\to\infty$.
The problem of extending the Burkholder--Davis--Gundy inequality to $2$-smooth Banach spaces has been considered by many authors \cite{Brz95, Brz97, Brz3, Dett89, Dett91, Ondrejat04}. The optimal asymptotic dependence of the constant in these inequalities for $p\to \infty$
was first obtained Seidler \cite{Sei}, who proved the following result.

\begin{theorem}[Seidler]\label{thm:Seid} Let $X$ be a $(2,D)$-smooth Banach space and let
$0<p<\infty$. For all $\Phi\in L_{\bF}^p(\Om;L^2(0,T;\gamma(H,X)))$
the process $(\int_0^t \Phi_s \ud W_s)_{t\in [0,T]}$ has a continuous version which satisfies
$$ \E \sup_{t\in [0,T]} \Big\n \int_0^t \Phi_s \ud W_s\Big\n^p\le C_{p,D}^p \|\Phi\|_{L^p(\Omega;L^2(0,T;\gamma(H,X)))}^p,
$$
where $C_{p,D}$ is a constant only depending on $p$ and $D$.
For $2\le p<\infty$ one may take $C_{p,D} = C_D \sqrt{p}$, where $C_D$ is a constant only depending on $D$.
\end{theorem}
The proof is based on an extension to $2$-smooth Banach spaces of the classical Burkholder--Rosenthal inequality due to Pinelis \cite{Pin}. Tracking and optimising constants in this reference one finds that the choice $C_D = 10D$ will do (see \cite{NV20} for the details).

\section{Maximal inequalities for stochastic convolutions}\label{sec:maxstochasticconv}

A family $(S(t,s))_{0\le s\le t\le T}$ of bounded linear operators on a Banach space
$X$ is called a {\em $C_0$-evolution family} indexed by $[0,T]$ if the following conditions are satisfied:
\begin{enumerate}[(1)]
\item $S(t,t) = I$  for all $t\in [0,T]$;
\item $S(t,r) = S(t,s) S(s,r)$ for all $0\le r\le s\le t\le T$;
\item the mapping $(t,s) \to S(t,s)$ is strongly
continuous on the set $\{0\le s\le t\le T\}$.
\end{enumerate}

Under the assumption that the Banach space $X$ is $(2,D)$-smooth, for processes $g \in L_{\bF}^0(\Om;L^2(0,T;\gamma(H,X)))$
we consider the stochastic convolution process $(u_t)_{t\in [0,T]}$ defined by
%%%
\begin{align*}%\label{eq:defug}
 u_t:= \int_0^t S(t,s)g_s \ud W_s, \qquad t\in [0,T].
 \end{align*}
As explained in the Introduction, the nomenclature ``stochastic convolution'' is justified by the important special case where the evolution family arises from a semigroup of operators.

The remainder of this paper is dedicated to surveying the following two problems:
\begin{itemize}
 \item to find conditions guaranteeing that $u$ has a continuous version which satisfies the Burkholder type $L^p$-maximal inequality
\begin{align}\label{eq:Lpmaxest}
\E\sup_{t\in [0,T]}\n u_t\n^p\leq C_{p,X}^p \|g\|_{L^p(\Omega;L^2(0,T;\gamma(H,X)))}^p;
\end{align}
\item if this is the case, to determine whether the constant $C_{p,X}$ is of order $O(\sqrt{p})$ as $p\to\infty$.
\end{itemize}
In \eqref{eq:Lpmaxest} and in the rest of the paper, we do not distinguish notationally between $u$ and its continuous version.
The right hand side of \eqref{eq:Lpmaxest} is motivated by Theorem \ref{thm:Seid}, which gives \eqref{eq:Lpmaxest}
in the special case of the trivial family $S(t,s) \equiv I$ with $O(\sqrt{p})$ dependence  of the constant as $p\to\infty$.

A number of general remarks can be made at this point.

\begin{remark}
In many applications the evolution family is generated by a family $(A(t))_{t\in [0,T]}$ of closed operators on $X$, in the sense  made precise in Subsection \ref{sec:maxineq-contr}\ref{subsec:Ito}. In this case the process $u$ can be interpreted as the mild solution to the stochastic differential equation
\begin{align}\label{eq:SDE}
\ud u_t = A(t) u_t \ud t + g_t \ud W_t, \qquad u(0) = 0.
\end{align}
If $u$ is a {\em strong solution} of \eqref{eq:SDE}, i.e., if for all $t\in [0,T]$ one has that $t\mapsto A(t) u_t$ belongs to $ L^1(0,t;X)$ almost surely and
\begin{align}\label{eq:strongsol}
u_t = \int_0^t A(s) u_s \ud s +\int_0^t  g_s  \ud W_s \ \ \ \hbox{almost surely}
\end{align}
(by the stochastic Fubini theorem this happens, e.g.,
when $u_t$ is $\Dom(A(t))$-valued and both $u$ and $Au$ belong to $L_{\bF}^0(\Omega;L^1(0,T;X))$),
then it is easy to see that $u$ has a continuous version, namely the process defined by the right-hand side of \eqref{eq:strongsol} once a continuous version of the stochastic integral $\int_0^t  g_s  \ud W_s$ has been selected.
\end{remark}

\begin{remark}\label{rem:extrapol}
If $u$ has a version satisfying \eqref{eq:Lpmaxest} for all $g\in L_{\bF}^p(\Omega;L^q(0,T;\gamma(H,X)))$,
for certain fixed $0<p<\infty$ and $1\le q\le \infty$,
a standard localisation argument shows that for all $g\in L_{\bF}^0(\Omega;L^q(0,T;\gamma(H,X)))$ the process $u$ has a continuous version. Moreover, an application of Lenglart's inequality \cite[Proposition IV.4.7]{RY} implies that \eqref{eq:Lpmaxest} (with $p$ replaced by $r$) extends to all exponents $0<r\le p$.
\end{remark}

For general $C_0$-evolution families, and even for $C_0$-semigroups, the problem of proving
the existence of a continuous version is open even when $X$ is a Hilbert space. In Subsections \ref{sec:maxstochasticconv}\ref{subsec:factorisation} and \ref{sec:maxstochasticconv}\ref{subsec:dilations}
we will discuss two techniques to approach this problem: the factorisation method of Da Prato, Kwapie\'n, and Zabczyk, and the
dilation method of Hausenblas and Seidler. Both methods also lead to maximal inequalities. In the case of the factorisation method
this inequality is weaker than \eqref{eq:Lpmaxest}; the dilation method gives \eqref{eq:Lpmaxest} with optimal asymptotic dependence of the constant.
In Section \ref{sec:maxineq-contr} we will see that for $C_0$-evolution families of {\em  contractions}, a continuous version always exists and \eqref{eq:Lpmaxest} holds with optimal asymptotic dependence of the constant.

One of the reasons for insisting on asymptotic $O(\sqrt{p})$-dependence of the constant is that it implies Gaussian tail estimates.
This is an immediate consequence of the special case $\alpha=2$ of following elementary lemma.

\begin{lemma}\label{lem:expesttail}
Let $\xi$ be a non-negative random variable and suppose there exist $\alpha>0$ and $C\ge 1$ such that
$\E\xi^p\leq C^p p^{p/\alpha}$ for all $p\geq \alpha.$
Then setting  $\sigma^2= e C^{\alpha}$ one has
$$\P(\xi\geq r) \leq 3\exp(-r^\alpha/(\alpha \sigma^2)), \qquad r>0.$$
\end{lemma}
\begin{proof}
By Markov's inequality,
$\P(\xi\geq r)\leq r^{-p} \E\xi^p\leq (C/r)^p p^{p/\alpha}.$
If $e^{-1}(r/C)^{\alpha}\geq \alpha$ we can set $p = e^{-1}(r/C)^{\alpha}$ to obtain $\P(\xi\geq r) \leq e^{-p/\alpha} = \exp(-r^\alpha/(\alpha \sigma^2))$.
If $e^{-1}(r/C)^{\alpha}<\alpha$, then $\P(\xi\geq r)\leq 1\leq 3e^{-1}\leq 3\exp(-r^\alpha/(\alpha \sigma^2))$.
\end{proof}

Indeed, applying the lemma to $\xi = \sup_{t\in [0,T]}\n \int_0^t S(t,s)g_s \ud W_s\n$ and $\alpha=2$ we obtain the following general result:

\begin{corollary}\label{cor:contractionS-new}
Let $(S(t,s))_{0\le s\le t\le T}$ be a $C_0$-evolution family of contractions on a $(2,D)$-smooth Banach space $X$
and let $g\in L_{\bF}^\infty(\Omega;L^q(0,T;\gamma(H,X)))$ with  $1\le q\le \infty$. If the maximal inequality
\[\E\sup_{t\in [0,T]}\Big\n \int_0^t S(t,s)g_s \ud W_s\Big\n^p\leq C^p (\sqrt{p})^p \|g\|_{L^\infty(\Omega;L^q(0,T;\gamma(H,X)))}^p\]
holds for all $2\le p<\infty$, where $C$ is a constant independent of $p$,
then the process $(\int_0^t S(t,s)g_s \ud W_s)_{t\in [0,T]}$ has a continuous version which satisfies the Gaussian tail estimate
\[\P\Bigl(\sup_{t\in [0,T]}\Big\n \int_0^t S(t,s)g_s \ud W_s\Big\n\geq r\Bigr) \leq 2\exp\Bigl(-\frac{r^2}{2\sigma^2}\Bigr), \qquad r>0,\]
where $\sigma^2 = eC^2\n g\n_{L^\infty(\Omega;L^q(0,T;\gamma(H,X)))}^2$.
\end{corollary}

This method of getting Gaussian tail estimates gives rather poor bounds on the variance. In Subsection \ref{subsec:Ito} we will discuss another method which, when applied to Theorem \ref{thm:contractionS-new}, gives a bound that is close to being optimal.

\subsection{The factorisation method}\label{subsec:factorisation}

The so-called factorisation method was introduced by Da Prato, Kwapie\'n, and Zabczyk \cite{DPKZ} to prove the existence of
a continuous version for stochastic convolutions with $C_0$-semigroups defined on a Hilbert space and was extended
to $C_0$-evolution families by Seidler \cite{Sei93}.
It is based on the formula
$$\int_{r}^t (t-s)^{\alpha-1} (s-r)^{-\alpha} \ud s = \frac{\pi}{\sin(\pi\alpha)},$$
from which one deduces the following identity, valid for $0<\alpha<\frac12$:
$$\frac{\pi}{\sin \pi \alpha}\int_0^t S(t,s) g_s \ud W_s =
\int_0^t (t-s)^{\alpha-1} S(t,s) \Bigl( \int_0^s (s-r)^{-\alpha}S(s,r) g_r \ud W_r\Bigr)\ud r.$$
For $2<p<\infty$ and $\frac1p<\alpha<\frac12$ the process $R_\alpha(s):= \int_0^s (s-r)^{-\alpha}S(s,r) g_r \ud W_r$ belongs to
$L^p(0,T;L^p(\Om;X))$, which we identify with $L^p(\Omega;L^p(0,T;X))$, and then use the fact that the mapping
$ R_\alpha\mapsto \int_0^t (t-s)^{\alpha-1} S(t,s) R_\alpha(s)\ud s$ maps the latter space into $L^p(\Om;C([0,T];X))$.
{\em Mutatis mutandis} this method extends to the more general setting of $2$-smooth Banach spaces. By bookkeeping the norm estimates and tracking constants, and performing a standard localisation argument, the following result is obtained.

\begin{theorem}[Factorisation]\label{thm:fact} Let $(S(t,s))_{0\le s\le t\le T}$ be a $C_0$-evolution family on a $(2,D)$-smooth Banach space $X$. For all $g\in L_{\bF}^0(\Omega;L^q(0,T;\gamma(H,X)))$ with $2<q<\infty$
the process $(\int_0^t S(t,s)g_s \ud W_s)_{t\in [0,T]}$ has a continuous version.
For $g\in L_{\bF}^p(\Omega;L^q(0,T;\gamma(H,X)))$ with $0<p\le q$, this version satisfies
\begin{align*}
\E\sup_{t\in [0,T]}\Big\n \int_0^t S(t,s)g_s \ud W_s\Big\n^p \leq C_{p,q,D,T}^p M^p\|g\|_{L^p(\Omega;L^q(0,T;\gamma(H,X)))}^p,
\end{align*}
where $M = \sup_{0\le s\le t\le T}\|S(t,s)\|$.
For $p=q$ one may take $C_{p,p,D,T} = D K_{p}\sqrt{p}T^{\frac{1}{2}-\frac{1}{p}}$ with $\limsup_{p\to \infty} K_p<\infty$.
\end{theorem}

It is important to observe that the estimate is phrased in terms of the norm of $L^p(\Omega;L^q(0,T;\gamma(H,X)))$, rather than
 $L^p(\Omega;L^2(0,T;\gamma(H,X)))$ as in the Burkholder type estimate \eqref{eq:Lpmaxest}. On the other hand, in contrast to the results of Section \ref{sec:maxineq-contr} where contractivity is required, Theorem \ref{thm:fact} is applicable to arbitrary $C_0$-evolution families.

\subsection{The dilation method}\label{subsec:dilations}

In this subsection we discuss an abstract version of a dilation technique due to the Hausenblas and Seidler \cite{HauSei1, HauSei2}.
In their original formulation for $C_0$-contraction semigroups on Hilbert spaces, the key idea is
to use the Sz.-Nagy dilation theorem \cite{Nagybook} to dilate the semigroup to a unitary $C_0$-group $(U(t))_{t\in \R}$ on a larger Hilbert space. Extending $g$ to this larger Hilbert space as well and using the group property to write
$$ \int_0^t U(t-s)g_s \ud W_s = U(t)\int_0^t U(-s)g_s \ud W_s,$$
the stochastic integral on right-hand side can be estimated by means of Theorem \ref{thm:Seid}, or rather, its special case for Hilbert spaces $X$. This then gives the result. Still in the setting of Hilbert spaces $X$, the method can be extended {\em mutatis mutandis} to the situation where $g \ud W$ is replaced by an arbitrary $X$-valued continuous local martingale.

There is no obvious way to extend the Hausenblas--Seidler argument to general $C_0$-semigroups or to $C_0$-evolution families. Moreover, the Sz.-Nagy dilation theorem is a Hilbert space theorem. To overcome both problems, the next definition introduces an abstract dilation framework.

\begin{definition}\label{def:dilation}
A $C_0$-evolution family $(S(t,s))_{0\leq s\leq t\leq T}$ on a Banach space $X$ is said to:
\begin{enumerate}[\rm (1)]
\item {\em admit an invertible dilation on the Banach space $Y$}, if there exist strongly continuous functions $J:[0,T]\to \calL(X,Y)$ and $Q:[0,T]\to \calL(Y,X)$ such that
    \[S(t,s) = Q(t) J(s) \ \ \ \text{for all} \ \ 0\le s\le t\le T.\]
\item {\em admit an approximate invertible dilation on the sequence of Banach spaces $(Y_n)_{n\geq 1}$}, if there exist strongly continuous functions $J_n:[0,T]\to \calL(X,Y_n)$ and $Q_n:[0,T]\to \calL(Y_n,X)$ such that 
    \[\sup_{n\geq 1}\sup_{t\in [0,T]}\|J_n(t)\|<\infty, \qquad \sup_{n\geq 1}\sup_{t\in [0,T]}\|Q_n(t)\|<\infty,\]
    and
    \[S(t,s)x = \lim_{n\to \infty} Q_n(t) J_n(s)x \ \ \ \text{for all} \ \ 0\le s\le t\le T \ \ \hbox{and} \ \ x\in X.\]
\end{enumerate}
\end{definition}

\begin{example}\label{ex:extransf1}
A sufficient condition for the existence of an invertible dilation is that every operator $S(t,s)$ be invertible, in which case we can take $Y=X$, $Q(t) = S(t,0)$, and $J(s) = S(s,0)^{-1}$.
\end{example}
\begin{example}\label{ex:extransf2} A $C_0$-semigroup $(S(t))_{t\ge 0}$ is said to {\em dilate to a $C_0$-group}
if there exist a $C_0$-group $(U(t))_{t\in \R}$ on a Banach space $Y$ and bounded operators $J\in \calL(X,Y)$ and $Q\in \calL(Y,X)$  such that $S(t) = Q U(t)J$ for all $t\ge 0$. In this case the operators $Q(t) := Q U(t)$ and $J(s) := U(-s) J$ define an invertible dilation in the sense of Definition \ref{def:dilation}.
In cases of interest, it is often possible to construct group dilations which preserve certain features of interest:

\begin{itemize}
\item If $(S(t))_{t\ge 0}$ is a $C_0$-semigroup of contractions on a Hilbert space $X$, then a unitary group dilation exists on a Hilbert space $Y$. This is the content of the Sz.-Nagy dilation theorem.
\item If $(S(t))_{t\ge 0}$ is a $C_0$-semigroup of positive contractions on an $L^p$-space with $1<p<\infty$, then a group dilation of positive contractions exists on another $L^p$-space. This is the content of Fendler's theorem \cite{Fen}.
\item If the negative generator $-A$ has a bounded $H^\infty$-calculus on of angle $<\frac12\pi$ on any Banach space $X$,
then a group dilation exists on the Banach space $\gamma(L^2(\R),X)$. This result is essentially due to \cite{FW} and stated in its present form in \cite{HNVW4}. If $X$ is $2$-smooth, then so is $\gamma(L^2(\R),X)$.
\end{itemize}

Further dilation results can be found in \cite{AFLM, FaGl, LeMerdy96}.
As far as we know, no extensions of these results are known for evolution families.
We also do not know whether every $C_0$-semigroup has an (approximate) invertible dilation in the sense of  Definition \ref{def:dilation}, or whether in the cases that such a dilation exists there also exists a group dilation.
Here it is important that the space $Y$ should enjoy similar geometric properties as $X$, such as Hilbertianity, $2$-smoothness, or  UMD.
\end{example}

\begin{example} We now give an example where an approximate dilation can be constructed.
Let $X$ and $X_1$ be Hilbert spaces, with $X_1$ continuously and densely embedded in $X$, and let $A\in C([0,T];\calL(X_1,X))$ be such that there exist constants $c>0$ and $C\ge 0$ such that
\[c\|x\|_{X_1} \leq \|x\|_X + \|A(t)x\|_X\leq C\|x\|_{X_1}, \quad t\in [0,T], \ x\in X_1.\]
Suppose further that for all $t\in [0,T]$ the operator $A(t)$ generates a $C_0$-contraction semigroup $(S^t(s))_{s\geq 0}$ and that for all $s_0, s_1, ,t_0, t_1\in [0,T]$ the operators $S^{t_0}(s_0)$ and $S^{t_1}(s_1)$ and their adjoints commute.
Then $A$ generates a $C_0$-evolution family $(S(t,s))_{0\le s\le t\le T}$ of contractions on $X$
in the sense of \cite[Theorem 5.3.1]{Pazy} or \cite[Theorem 4.4.1]{Ta1}.

Setting $t_k^n = \frac{kT}{n}$ and $I_k^n = [t_k^n, t_{k+1}^n)$ (with endpoint included if $k=n-1$),
from the proof of the theorems just cited one infers
$S(t,s)x = \lim_{n\to \infty} S_n(t,s)x$ for all $0\leq s\leq t\leq T$ and $x\in X$, where
\[S_n(t,s) = \left\{
                \begin{array}{ll}
                  S^{t_k^n}(t-s), & \hbox{$s,t\in I_k^n$;} \\
                  S^{t_\ell^n}(t-t_\ell^n)\Big(\prod_{j=k+1}^{\ell-1} S^{t_j^n}(T/n)\Big)S^{t_{k}^n}(s-t_k^n),
& \hbox{$s\in I_{k}^n, t\in I_{\ell}^n, k\le \ell$.}
                \end{array}
              \right.
 \]
It is easy to check that $(S_n(t,s))_{0\le s\le t\le T}$ is a $C_0$-evolution family of contractions. By the assumption that the contraction semigroups $(S^t(s))_{s\ge 0}$ commute
among themselves and with their adjoints, it follows from \cite[Proposition 9.2]{Nagybook}
that there exist a Hilbert space $Y$ and contractions $J\in \calL(X,Y)$ and $Q\in \calL(Y,X)$, as well as commuting isometric $C_0$-groups $(U^t(s))_{s\in \R}$ on $Y$ such that for all $s_1, \ldots, s_n$ and $ t_1, \ldots, t_n\in [0,T]$ we have
\begin{align}\label{eq:dilationcomm}
 S^{t_1}(s_1) \ldots S^{t_n}(s_n) = Q  U^{t_1}(s_1) \ldots U^{t_n}(s_n) J.
\end{align}
For  for $0\le s\le t\le T$ we define the operators $U_n(t,s)$ by
\[U_n(t,s) := \left\{
                \begin{array}{ll}
                  U^{t_k^n}(t-s), & \hbox{$s,t\in I_k^n$;} \\
                  U^{t_\ell^n}(t-t_\ell^n)\Big(\prod_{j=k+1}^{\ell-1} U^{t_j^n}(T/n)\Big)U^{t_{k}^n}(s-t_k^n),
& \hbox{$s\in I_{k}^n, t\in I_{\ell}^n, k\le \ell$.}
                \end{array}
              \right.
 \]
Then $(U_n(t,s)_{0\le s\le t\le T}$ is $C_0$-evolution family of invertible operators, and by \eqref{eq:dilationcomm} we have
\[ S_n(t,s) = Q  U_n(t,s)J, \qquad 0\leq s\leq t\leq T.\]
It follows that there exists an approximate invertible dilation given by
$Q_n(t) = Q U_n(t,0)$ and $J_n(s) = U_n(s,0)^{-1} J$.
\end{example}

The next theorem extends the Hausenblas--Seidler dilation theorem to evolution families on $2$-smooth Banach spaces.
By Example \ref{ex:extransf2} it is applicable to $C_0$-semigroups on $2$-smooth Banach spaces whose negative generator has a bounded
$H^\infty$-calculus of angle $<\frac12\pi$. The resulting maximal inequality, with $O(\sqrt{p})$ dependence of the constant  as $p\to \infty$, was obtained independently in \cite{Sei} and \cite{VerWei}. Some of the maximal estimates in the latter paper are valid for a class of processes strictly larger than $L_{\bF}^0(\Omega;L^2(0,T;\gamma(H,X)))$, but with best-known constant of order $O(p)$ instead of $O(\sqrt{p})$.

\begin{theorem}[Dilation]\label{thm:approxinvdilation}
Let $(S(t,s))_{0\leq s\leq t\leq T}$ be a $C_0$-evolution family on a $(2,D)$-smooth Banach space $X$ which admits an approximate invertible dilation on a sequence of $(2,D)$-smooth Banach spaces $(Y_n)_{n\geq 1}$.
For all $0< p<\infty$ and $g\in L_{\bF}^p(\Omega;L^2(0,T;\gamma(H,X)))$ the process $(\int_0^t S(t,s)g_s \ud W_s)_{t\ge 0}$ has a continuous version which satisfies
\[\E\sup_{t\in [0,T]}\Big\n \int_0^t S(t,s)g_s \ud W_s\Big\n^p
\leq C_{p,D}^p C_J^p C_Q^p\|g\|_{L^p(\Omega;L^2(0,T;\gamma(H,X)))}^p,\]
where $C_J = \sup_{n\geq 1}\sup_{t\in [0,T]}\|J_n(t)\|$ and  $C_Q = \sup_{n\geq 1}\sup_{t\in [0,T]}\|Q_n(t)\|$.
For $2\le p<\infty$ one may take $C_{p,D}  = 10D \sqrt{p}$.
\end{theorem}
\begin{proof}
By Remark \ref{rem:extrapol} it suffices to consider the case $2\le p<\infty$, and by a limiting argument it even suffices to consider the case $2<p<\infty$.

Let us first assume that  $g\in L_{\bF}^p(\Omega;L^p(0,T;\gamma(H,X)))$.
For such processes, Theorem \ref{thm:fact}
implies the existence of a continuous version.
By monotone convergence it suffices to prove the maximal estimate with suprema taken over finite sets $\pi\subseteq [0,T]$. For $t\in \pi$, write
\begin{align*}
u_t & = \int_0^t S(t,s) g_s  \ud W_s = \limn \int_0^t Q_n(t) J_n(s) g_s  \ud W_s = \limn Q_n(t) \int_0^t  J_n(s) g_s  \ud W_s,
\end{align*}
where the limit is taken in  $L^p(\Omega;X)$ by dominated convergence. Using that $\pi$ is finite, by taking suitable subsequences we may assume the above limit holds pointwise on $\pi\times \Omega_0$, where $\Omega_0\subseteq\Om$ is a measurable set with $\P(\Omega_0) = 1$.
Therefore, for all $t\in \pi$, pointwise on $\Omega_0$ we have
\begin{equation*}
\|u_t\|_{X} \leq C_Q \liminf_{n\to \infty} \Big\|\int_0^t J_n(s) g_s  \ud W_s\Big\|_{Y_n}.
\end{equation*}
Taking the supremum over $t\in \pi$, upon taking $L^p(\Omega)$-norms we obtain
\begin{align*}
\E\sup_{t\in \pi}\|u_t\|^p &\leq C_Q^p \E \sup_{t\in \pi} \liminf_{n\to \infty} \Big\|\int_0^t J_n(s) g_s  \ud W_s\Big\|_{Y_n}^p
\\ &\leq C_Q^p \E\liminf_{n\to \infty} \sup_{t\in \pi} \Big\|\int_0^t J_n(s) g_s  \ud W_s\Big\|_{Y_n}^p
\\ &\leq C_Q^p \liminf_{n\to \infty} \E \sup_{t\in \pi} \Big\|\int_0^t J_n(s) g_s  \ud W_s\Big\|_{Y_n}^p & \text{by Fatou's Lemma}
\\ &\leq (10D  \sqrt{p} C_Q)^p \liminf_{n\to \infty} \|J_n g\|_{L^p(\Omega;L^2(0,T;\gamma(H,Y_n)))}^p & \text{by Theorem \ref{thm:Seid}}
\\ &\leq (10D \sqrt{p} C_J  C_Q)^p\|g\|_{L^p(\Omega;L^2(0,T;\gamma(H,X)))}.
\end{align*}
This gives the result for processes $g\in L_{\bF}^p(\Omega;L^p(0,T;\gamma(H,X)))$.
The general case of processes $g\in L_{\bF}^p(\Omega;L^2(0,T;\gamma(H,X)))$ follows from it by approximation.
\end{proof}

Remarkably, the method of dilations has been
used \cite{PesZab, SchnVer17} to derive maximal inequalities also for the case of stochastic Volterra equations on Hilbert spaces.

\section{The contractive case}\label{sec:maxineq-contr}

Up to this point we have considered general $C_0$-evolution families. In the present section we take a closer look at
the special case of $C_0$-evolution families of contractions. By a standard rescaling argument, the results of this section extend to the situation where, for some $\lambda\geq 0$, one has
$$\|S(t,s)\| \le e^{\la(t-s)}, \qquad 0\le s\le t\le T.$$
An additional term $e^{\lambda T}$ has then to be added on the right-hand side of the estimates.

\subsection{The main result}

We begin with a general result on the existence of continuous versions. It extends a result stated in \cite{Kot82} for Hilbert spaces
and continuous square integrable martingales, to $(2,D)$-smooth Banach spaces and Brownian motion. Replacing Hilbertian $L^2$-estimates by \cite[Lemma 2.2]{NV20} and Proposition \ref{prop:Neid}, the original argument can be generalised and leads to the following result with $D^4$ instead of $D^2$; an additional approximation argument permits the passage to $D^2$.

\begin{proposition}\label{prop:contractionS}
Let $(S(t,s))_{0\le s\le t\le T}$ be a $C_0$-evolution family of contractions on a $(2,D)$-smooth Banach space $X$. For all $g\in L_{\bF}^0(\Omega;L^2(0,T;\gamma(H,X)))$
the process $(\int_0^t S(t,s)g_s \ud W_s)_{t\in[0,T]}$ has a continuous version.
If $g\in L_{\bF}^2(\Omega;L^2(0,T;\gamma(H,X)))$, then it satisfies the following tail estimate for all $r>0$:
\[\P\Bigl(\sup_{t\in [0,T]}\Big\n \int_0^t S(t,s)g_s \ud W_s\Big\n\geq r\Bigr) \leq \frac{D^2}{r^2} \|g\|_{L^2(\Om;L^2(0,T;\gamma(H,X)))}^2.\]
\end{proposition}

By combining the discretisation technique used in the proof of this proposition
with a version of a theorem of Pinelis \cite{Pin} used in the proof of Theorem \ref{thm:Seid}, the following
$L^p$-maximal inequality has been recently obtained in \cite{NV20}.

\begin{theorem}\label{thm:contractionS-new}
Let $(S(t,s))_{0\le s\le t\le T}$ be a $C_0$-evolution family of contractions on a $(2,D)$-smooth Banach space $X$.
For all $g\in L_{\bF}^p(\Omega;L^2(0,T;\gamma(H,X)))$ with $0<p<\infty$
the process $(\int_0^t S(t,s)g_s \ud W_s)_{t\in [0,T]}$ has a continuous version which satisfies
\[\E\sup_{t\in [0,T]}\Big\n \int_0^t S(t,s)g_s \ud W_s\Big\n^p\leq C_{p,D}^p \|g\|_{L^p(\Omega;L^2(0,T;\gamma(H,X)))}^p.\]
For $2\le p<\infty$ one may take $C_{p,D} =  10D \sqrt{p}$.
\end{theorem}

Theorem \ref{thm:contractionS-new} is in some sense definitive, in that it applies to arbitrary $C_0$-evolution families of contractions and gives the correct order $O(\sqrt{p})$ of the constant; as such it is new even for Hilbert spaces $X$. It is also new for $C_0$-semigroups of contractions in $2$-smooth Banach spaces.

Theorem \ref{thm:contractionS-new} has a long history with contributions by many authors.  Here we will only review the semigroup approach; $L^2$-maximal inequalities for monotone stochastic evolution equations with random coefficients in Hilbert spaces are older and go back to \cite{Par} and \cite{KryRoz}. For an exposition and further references to the literature the reader is referred to \cite{LiuRoc}.
The first author to use semigroup methods to derive $L^2$-maximal inequalities is Kotelenez \cite{Kot82} who obtained path continuity in the more general situation where term $g \ud W$ is replaced by an arbitrary continuous square integrable $X$-valued martingale. This paper also contains a weak type estimate similar to the one of Proposition \ref{prop:contractionS}. Still for Hilbert spaces and $p=2$, \eqref{eq:Lpmaxest} was first proved in \cite{Ich82, Kot84} using It\^o's formula applied to the $C^2$-function $x\mapsto \n x\n^2$ under further assumptions on the evolution family.
For $C_0$-contraction semigroups, these results were extended to exponents $2\le p<\infty$ by Tubaro \cite{Tub}, who applied It\^o's formula to the mapping $x\mapsto \n  x\n ^p$ which for Hilbert is twice continuously Fr\'echet differentiable. The extension to exponents $0<p<2$ was subsequently obtained by Ichikawa \cite{Ichi}.
Tubaro's method of proof was revisited by Brze\'zniak and Peszat \cite{BrzPes}, who extended it
to Banach spaces $X$ with the property that for some $2\le p<\infty$ the mapping $x\mapsto \n x\n ^p$ is twice continuously Fr\'echet differentiable and the first and second Fr\'echet derivatives are bounded by constant multiples of $\n x\n^{p-1}$ and $\n x\n^{p-2}$, respectively. Spaces with this property are $2$-smooth and include  $L^q(\mu)$ for $2\le q\le p<\infty$ and the Besov spaces $B^{s}_{q,r}(\R^d)$ for $2\leq q\leq r\leq p<\infty$. In the converse direction it is
not true that all $2$-smooth spaces satisfy the twice differentiability condition; an abstract conunterexample follows from \cite[Theorem 3.9]{LeoSun} (see also \cite[Example 1.1]{NeeZhu}). In the Besov scale the twice differentiability condition is unclear if $2\le r<q<\infty$, even though the space $B^{s}_{q,r}(\R^d)$ is $2$-smooth in that case, too.
The approach based on It\^o's formula extends to evolution families, but has the general drawback that it does not seem to give the optimal growth rate $O(\sqrt{p})$ of the constant as expected from the Burkholder--Davis--Gundy inequalities as $p\to \infty$.
As discussed in Subsection \ref{sec:maxstochasticconv}\ref{subsec:dilations}, for $C_0$-contraction semigroups on Hilbert spaces a new proof of the maximal inequality for exponents $0<p<\infty$ giving growth of order $O(\sqrt{p})$ was obtained by Hausenblas and Seidler \cite{HauSei1}.

The approach via It\^o's formula was once more revisited in \cite{NeeZhu}, where it was finally extended to arbitrary $2$-smooth Banach spaces by exploiting the fact that, in such spaces, for $2\le p<\infty$ the mapping $x\mapsto \n  x\n ^p$ is once continuously Fr\'echet differentiable with a Lipschitz continuous derivative. As it turns out, this already suffices to prove a version of the It\^o formula with the help of which the argument can be completed. This approach, however, does not seem to give the optimal $p$-dependence of the constant as $p\to\infty$.

\subsection{The It\^o formula approach revisited}\label{subsec:Ito}

The aim of the present subsection is to present the It\^o formula approach to maximal estimates for stochastic convolutions.
In comparison with Theorem \ref{thm:contractionS-new} it does not lead to new results (in fact we need stronger assumptions on the evolution family and obtain non-optimal asymptotic dependence of the constant), but this approach has the merit that it can be extended to {\em random} $C_0$-evolution families of contractions. To the best of our knowledge, for this setting no maximal $L^p$-estimates of the form \eqref{eq:Lpmaxest} in $2$-smooth spaces are available in the literature.
For stochastic evolution equations with random coefficients in Hilbert spaces subject to monotonicity conditions, $L^2$-maximal inequalities go back to \cite{Par} and \cite{KryRoz}; for an exposition and further reference see \cite{LiuRoc}. Some extensions to the case $p\not=2$ have been obtained recently in \cite{NeeSis}.

In order to avoid technicalities that would obscure the line of argument we present our main results for non-random evolution families and indicate the changes that have to be made in the $\Omega$-dependent case in Remark \ref{rem:Aadapted}.
Rather than discussing the maximal $L^p$-inequality in \cite{NeeZhu} in detail, we will provide a detailed proof of a Gaussian tail estimate. The rationale of this choice is that this estimate cannot be deduced (e.g., via Lemma \ref{lem:expesttail}) from the result of \cite{NeeZhu} due to the fact it doesn't provide the correct order $O(\sqrt{p})$ of the constant. The result presented here is new, in that it generalises \cite[Theorem 1.2]{BrzPes} to arbitrary $2$-smooth Banach spaces. A further novel feature of our result is that it gives an improved bound on the variance.

As in \cite{BrzPes, ZhuBrzLiu} the idea is to apply It\^o's formula to $h_\la:X\to [0,\infty)$ given by
 $$ h_\la(x) := (1+\la \n x\n^2)^{1/2}, \qquad x\in X.$$
The function $h_\la$ is Fr\'echet differentiable and
\[h_\la'(x) = \frac{\la q'(x)}{(1+\la \n x\n^2)^{1/2}},\]
where $q(x) := \|x\|^2$ is known to be Fr\'echet differentiable (see \cite{NeeZhu}) with
$q'(0)=0$ and
\begin{align}\label{eq:estqaccentx}
q'(x) = 2\n x\n n_x, \qquad x\not=0,
\end{align}
where $n_x$ is the Fr\'echet derivative of $\n \cdot\n$ at $x\not=0$ and satisfies $\n n_x\n =1$.
Although $h_{\lambda}$ is generally not $C^2$, the following `It\^o inequality' holds:

\begin{theorem}\label{thm-Ito} Let $X$ be a $2$-smooth Banach space and let $(a_t)_{t\in [0,T]}$ and $(g_t)_{t\in [0,T]}$ be processes in $L_{\bF}^0(\Om;L^1(0,T;X))$ and $ L_{\bF}^0(\Om;L^2(0,T;\gamma(H,X)))$, respectively. Fix $x\in X$ and let the process $(\xi_t)_{t\in [0,T]}$ be given by
    \begin{align*}
   \xi_t:=x+\int_0^t a_s\ud  s+\int_0^t g_s   \ud  W_s.
\end{align*}
Then, almost surely, for all $t\in [0,T]$ we have
 \begin{align}\label{eq:Itohla}
h_\la(\xi_t)\leq h_\la(x)+\int_0^t \lb a_s,h_\la'(\xi_s)\rb\ud  s+\int_0^t h_\la'(\xi_s)\circ g_s \ud  W_s + \frac12 D^2\la \|g\|_{L^2(0,t;\gamma(H,X))}^2.
\end{align}
\end{theorem}
\begin{proof}
We proceed in three steps.

\smallskip
{\em Step 1}. \ First suppose that $a$ and the operators in the range of $g$ take values in a fixed finite-dimensional subspace $Y$ of $X$. Then $\xi$ also takes its values in $Y$. Now we regularise the norm as in \cite[Lemma 2.2]{Pin}.  Let $\mu$ be a centred Gaussian measure with support $\text{supp}(\mu) = Y$. Fix $\varepsilon>0$ and let $q_{\varepsilon}:Y\to \R$ be given by
$$q_{\varepsilon}(x):= \int_Y \|x-\varepsilon y\|^2 \ud \mu(y).$$
Then by \cite[Lemma 2.2]{Pin} the function $q_{\varepsilon}$ has Fr\'echet derivatives of all orders, and
\begin{align}\label{eq:Pinelisestqvareps}
\big|q_{\varepsilon}(x)^{1/2} - \|x\|\big|\leq \varepsilon, \ \  \|(q_{\varepsilon}^{1/2})'(x)\| \leq 1, \ \ q_{\varepsilon}''(x)(v,v)\leq 2D^2\|v\|^2.
\end{align}
Moreover, $q_{\varepsilon}'(x) = \int_{Y} q'(x-\varepsilon y) \ud \mu(y)$ and from \eqref{eq:estqaccentx} and the dominated convergence theorem we obtain $q_{\varepsilon}'(x)\to q'(x)$ as $\varepsilon\downarrow 0$. Writing $q_\eps = q_\eps^{1/2}q_\eps^{1/2}$, differentiation by the product rule gives
\begin{align}\label{eq:qepsilonest}
\|q_{\varepsilon}'(x)\| = 2\|(q_{\varepsilon}^{1/2})'(x)\| q_{\varepsilon}^{1/2}(x)\leq 2q_{\varepsilon}^{1/2}(x).
\end{align}
It follows that the function $h_{\lambda,\varepsilon}:Y\to \R$ given by $$h_{\lambda,\varepsilon}(x) := (1+\lambda q_{\varepsilon}(x))^{1/2}, \qquad x\in Y,$$
has Fr\'echet derivatives of all orders, and
\[h_{\la,\varepsilon}'(x) = \frac{\lambda q_{\varepsilon}'(x)}{2(1+\lambda
q_{\varepsilon}(x))^{1/2}}, \ \  \ \ h_{\la,\varepsilon}''(x)(y,y) = \frac{\lambda q_{\varepsilon}''(x)(y,y)}{2(1+\lambda q_{\varepsilon}(x))^{1/2}} - \frac{\lambda^2 \lb y, q_{\varepsilon}'(x)\rb^2}{4(1+\lambda q_{\varepsilon}(x))^{3/2}}.\]
Therefore, by \eqref{eq:Pinelisestqvareps} and \eqref{eq:qepsilonest} for all $x\in Y$ one has $h_{\lambda,\varepsilon}(x)\to h_{\lambda}(x)$, $h_{\lambda,\varepsilon}'(x)\to h_{\lambda}'(x)$ as $\varepsilon\downarrow 0$, and
\begin{align}\label{eq:hlambdaepsilonest}
\|h_{\lambda,\varepsilon}'(x)\|\leq \sqrt{\lambda}, \qquad h_{\lambda,\varepsilon}''(x)(y,y)\leq D^2 \lambda \|y\|^2.
\end{align}

{\em Step 2}. \
By the It\^o formula,
\begin{equation}\label{eq:Itohlambdaeps}
\begin{aligned}
h_{\la,\varepsilon}(\xi_t) = h_{\la,\varepsilon}(x)+\int_0^t \lb a_s, h_{\la,\varepsilon}'(\xi_s)\rb\ud  s & +\int_0^t h_{\la,\varepsilon}'(\xi_s)\circ g_s\ud  W_s \\ &  + \frac12\int_0^t h_{\lambda,\varepsilon}''(\xi_s)(g_s , g_s ) \ud s.
\end{aligned}
\end{equation}
Since
\begin{align}\label{eq:quadr}\frac12\int_0^t h_{\lambda,\varepsilon}''(\xi_s)(g_s , g_s ) \ud s\leq \frac12 D^2\la \|g\|_{L^2(0,t;\gamma(H,X))}^2 \ \ \ \hbox{almost surely}
\end{align}
this proves \eqref{eq:Itohla} with $h_{\lambda,\varepsilon}$ instead of $h_{\lambda}$.

It remains to let $\varepsilon\downarrow 0$ in each of the terms in \eqref{eq:Itohlambdaeps}, except the last one which is estimated using \eqref{eq:quadr}. By path-continuity of the integrals it suffices to prove convergence for every fixed $t\in [0,T])$.

Convergence of the first two terms in \eqref{eq:Itohlambdaeps} is clear from the preliminaries in Step 1. For the third and fourth terms we can apply the pointwise convergence and the dominated convergence theorem (using the bound \eqref{eq:hlambdaepsilonest}) to obtain $\lb a_s, h_{\la,\varepsilon}'(\xi)\rb\to \lb a_s, h_{\la}'(\xi)\rb$ in $L^1(0,t)$ almost surely and
$h_{\la,\varepsilon}'(\xi)\circ g \to  h_{\la}'(\xi)\circ g$ in $L^2(0,t;H)$ almost surely to obtain the required convergence. This completes the proof in the finite-dimensional case.

\smallskip
{\em Step 3}. \ In the general case let $(a_n)_{n\geq 1}$ be a sequence simple functions and $(g_n)_{n\geq 1}$ be a sequence of finite rank adapted step processes such that $a_n\to a$ in $L^1(0,t;X)$ and $b_n\to b$ in $L^2(0,t;\gamma(H,X))$ almost surely. Let $\xi_n(t) := x+\int_0^t a_{n,s} \ud s + \int_0^t g_n \ud W$. Then $\xi_n\to \xi$ in $L^0(\Omega;C([0,t];X))$, and by passing to a subsequence we may suppose that $\xi_n\to \xi$ in $C([0,t];X)$ almost surely. By Step 1, \eqref{eq:Itohla} holds with $(a, g, \xi)$ replaced by $(a_n, g_n, \xi_n)$. Since $h_{\lambda}'$ is uniformly bounded and Lipschitz with constant $D^2 \lambda$ (this follows from the second estimate in \eqref{eq:hlambdaepsilonest} and letting $\eps\downarrow 0$), by dominated convergence we obtain $\lb a_n,h_{\lambda}'(\xi_n)\rb \to \lb a,h_{\lambda}'(\xi)\rb$ in $L^1(0,t;X)$ almost surely and $h_\la'(\xi_{n,s})\circ g_{n,s} \to h_\la'(\xi_s)\circ g_s $. Letting $n\to \infty$ we obtain \eqref{eq:Itohla} for $(a, g, \xi)$.
\end{proof}

For the remainder of the paper we assume that the following hypothesis is satisfied.

\begin{hypothesis}\label{hyp:evol} $(S(t,s))_{0\le s\le t\le T}$ is a $C_0$-evolution family of contractions and
$(A(t))_{t\in [0, T]}$ is a family of closed operators, acting on the same Banach space $X$. They enjoy the following properties:
\begin{enumerate}[\rm (1)]
\item\label{it:hyp1}  For all $t\in [0,T]$ we have
$(0,\infty)\subseteq \varrho(A(t))$ and there exist constant $M\geq 1$ such that
\[ \n \la (\la - A(t))^{-1}\n \le M, \ \  t\in [0,T],\ \  \lambda>0.\]
\item\label{it:hyp2} For all $t\in [0,T]$
and $\la\in (0,\infty)$ we have  $$\sup_{0\leq s\leq t\leq T}\|A(t) S(t,s) R(\lambda,A(s))\|<\infty.$$
\item\label{it:hyp3} For all $s\in [0,T]$ and $x\in \Dom(A(s))$ we have $S(\cdot,s)x\in W^{1,1}(s,T;X)$ and, for almost all $t\in [s,T]$,
\[\frac{{\rm d}}{{\rm d}t} S(t,s)x = A(t) S(t,s)x.\]
\end{enumerate}
\end{hypothesis}

\begin{remark}\label{rem:folk}
 It is folklore in the theory of evolution families that if Hypothesis \ref{hyp:evol} holds and each operator $A(t)$ is the generator of a $C_0$-semigroup, then \eqref{it:hyp1} holds with $M=1$.
\end{remark}

 Condition \eqref{it:hyp1} means that the operators $-A(t)$ are sectorial, uniformly with respect to $t\in[0,T]$. Condition \eqref{it:hyp2} expresses that $S(t,s)$ maps $\Dom(A(s))$ into $\Dom(A(t))$ with control on the norms uniformly with respect to $0\le s\le t\le T$. Condition \eqref{it:hyp3}  connects the operators $A(t)$ with $S(t,s)$ in the same way as a generator is connected to a semigroup of operators.
These conditions are satisfied in many applications (see e.g.\ \cite{AcqTer87, Amann95, Lun95, Pazy, Ta1, Yagi10}).

We are now ready to state the main result of this section. Under the additional assumption of Hypothesis \ref{hyp:evol} it provides another proof of the Gaussian tail estimate that can be obtained by combining Theorem \ref{thm:contractionS-new} with Corollary \ref{cor:contractionS-new}. The bound on the variance $\sigma^2$ obtained from that argument,
namely $100e D^2\n g\n_\infty^2$, is improved here to $2D^2 M^2\n g\n_\infty^2$, where $M$ is the constant in Hypothesis \ref{hyp:evol}\eqref{it:hyp1}. By Remark \ref{rem:folk} have $M=1$ in the case of a $C_0$-evolution family of contractions.

\begin{theorem}[Gaussian tail estimate]\label{thm:ExpestIto}
Let $(S(t,s))_{0\le s\le t\le T}$ be a $C_0$-evolution family of contractions on a $(2,D)$-smooth Banach space satisfying Hypothesis \ref{hyp:evol}.  For all $g\in L_{\bF}^\infty(\Om;L^2(0,T;\gamma(H,X)))$ the process $(\int_0^t S(t,s)g_s\ud W_s)_{t\in [0,T]}$ has a continuous version which satisfies
$$ \P\Bigl(\sup_{t\in [0,T]} \Big \n \int_0^t S(t,s)g_s\ud W_s\Big\n \ge r\Bigr) \le 3 \exp\Bigl(-\frac{r^2}{2\sigma^2}\Bigr)$$
for all $r>0$, where $\sigma^2 = 2D^2M^2\n g \n_{L^\infty(\Om;L^2(0,T;\gamma(H,X)))}^2$.
\end{theorem}

\begin{proof} The main idea is that
Theorem \ref{thm-Ito} provides the right estimate to generalise the proof of Brze\'zniak and Peszat's \cite[Theorem 1.2]{BrzPes}.
The proof will use some additional facts from stochastic analysis which are all standard and can be found in \cite{Kal, RY}.

\smallskip
{\em Step 1}. \
Let us first assume that $g\in L^\infty(\Om;L^2(0,T;\gamma(H,X)))$ is such that for every $t\in [0,T]$ we have $g_t\in \Dom(A(t))$ and $t\mapsto A(t)g_t$ belongs to $L^2(0,T;\gamma(H,X))$ almost surely. Under this assumption we claim that $u$ is a strong solution, i.e., almost surely we have
\begin{align}\label{eq:Ito-strong}
   u_t=\int_0^tA(s)u_s \ud s+\int_0^tg_s  \ud W_s, \qquad t\in [0,T].
\end{align}
This means that the assumptions of Theorem \ref{thm-Ito} are satisfied with $a_t = A(t)u_t$.
In order to prove \eqref{eq:Ito-strong} we set $u_t:= \int_0^t S(t,s)g_s\ud W_s$. Then
\[A(t) u_t =  \int_0^t A(t) S(t,s) g_s  \ud W_s\]
almost surely, since $A(t) S(t,s) R(1,A(s))$ is uniformly bounded by part \eqref{it:hyp2} of the hypothesis.
A standard argument involving the stochastic Fubini theorem and the formula
\[\int_r^t A(s) S(s,r) x\ud s=  S(t,r) x-x, \qquad x\in \Dom(A(r)),\]
(which follows from part \eqref{it:hyp3} of the hypothesis)
implies that for all $t\in [0,T]$ the identity \eqref{eq:Ito-strong} holds almost surely. By path continuity,
almost surely the identity holds for all $t\in [0,T]$. This concludes the proof of the claim.

\smallskip
{\em Step 2}. \
Since $ t\mapsto h_\la(S(t,s)x) = (1+\la \n S(t,s) x\n^2)^{1/2}$
is non-increasing by the contractivity of $S(t,s)$, for all $x\in \Dom(A(t))$ we have
$$\lb A(t)S(t,s)x,h_\la'(S(t,s)x)\rb = \frac{{\rm d}}{{\rm d}t} h_\la(S(t,s)x) \leq 0.$$
Therefore, setting $t=s$, we obtain $\lb A(s)x, h_\la'(x)\rb \leq 0$ for almost all $s\in [0,T)$.
Hence, by Theorem \ref{thm-Ito} applied with $a_s = A(s)u_s $ and $x=0$, and noting that $h_\la(0)=1$,
\begin{equation}\label{eq:hlaXt}
\begin{aligned}
 h_\la(u_t) \le
1+ \int_0^t h_\la'(u_s )\circ g_s  \ud W_s + \frac12 \lambda D^2 \|g\|_{L^2(0,t;\gamma(H,X))}^2.
\end{aligned}
\end{equation}
Below we will several times use that $\|h_\la'(x)\|\leq \lambda^{1/2}$  (see \eqref{eq:hlambdaepsilonest}).

The quadratic variation of the process defined by $N_t:= \int_0^t h_\la'(u_s)\circ g_s\ud W_s$ is given by $[N]_t = \int_0^t \n h_\la'(u_s)\circ g_s \n_H^2\ud s$. Therefore the process defined by
$Z_t := e^{N_t - \frac12[N]_t}$ is a local martingale by It\^o's formula and
\[Z_t = 1 + \int_0^t Z_s h_\la'(u_s )\circ g_s \ud W_s \]
Since $Z$ is non-negative, it is a supermartingale and therefore $\E(Z_t) \leq 1$. Since $[N]_t\leq \lambda \|g\|_{L^\infty(\Omega;L^2(0,T;\gamma(H,X)))}^2$ almost surely, it is standard to check $Z$ is actually a martingale with $\E(Z_t) = \E(Z_0) = 1$ for all $t\in [0,T]$ (this follows for instance from Novikov's condition).
We can rewrite \eqref{eq:hlaXt} in the form (using that $D\geq 1$)
\begin{align*}
h_\la(u_t) & \le 1+ \log Z_t +  \frac12\int_0^t \n h_\la'(u_s )\circ g_s  \n_H^2\ud s + \frac12\lambda D^2 \|g\|_{L^2(0,t;\gamma(H,X))}^2
\\ & \leq 1+ \log Z_t +  \lambda C_g,
\end{align*}
where $C_g:=D^2 \|g\|_{L^\infty(\Omega;L^2(0,T;\gamma(H,X)))}^2$. Setting $\phi_{\lambda}(r) = (1+\la r^2)^{1/2}$, Doob's inequality gives
\begin{align*}
\P\Bigl(\sup_{t\in [0,T]} \n u_t\n \ge r\Bigr)
& =  \P\Bigl(\sup_{t\in [0,T]}  h_\la(u_t) \ge \phi_\la (r)\Bigr)
\\ & =  \P\Bigl(\sup_{t\in [0,T]} \log Z_t \ge \phi_\la (r) -1 - \la C_g\Bigr)
\\ & =  \P\Bigl(\sup_{t\in [0,T]} Z_t \ge \exp(\phi_\la (r) -1 - \la C_g)\Bigr)
\\ & \le \exp(1 + \la C_g- \phi_\la (r)) \E Z_T
 = \exp(1 + \la C_g- \phi_\la (r)).
\end{align*}
If $r^2 \ge 2C_g$, choose $\la>0$ so that $1+\la r^2 = r^4/(4C_g^2)$. Then $\phi_\la(t) = r^2/(2C_g)$ and
 $$  \P\Bigl(\sup_{t\in [0,T]} \n u_t\n \ge r\Bigr)
 \le \exp\Bigl(1 + \frac{r^2}{4C_g^2} C_g- \frac{r^2}{2C_g}\Bigr)
 = \exp\Bigl(1 - \frac{r^2}{4C_g}\Bigr) \le 3  \exp\Bigl(- \frac{r^2}{4C_g}\Bigr).
$$
If $0<r^2< 2C_g$, we have the trivial inequality
$$  \P\Bigl(\sup_{t\in [0,T]} \n u_t\n \ge r\Bigr) \le 1 \le 3e^{-1/2} \le  3\exp\Bigl(- \frac{r^2}{4C_g}\Bigr).$$
This proves the result under the additional assumption on $g$ made at the beginning of Step 1.

\smallskip
{\em Step 3}. \
In the general case set $g_t^{(n)} := nR(n,A(t))g_t $ for $n\geq 1$. For all $t\in [0,T]$ and $\omega\in \Omega$ we have $\|g_t^{(n)}(\omega)\|\leq M\|g_t(\omega)\|$ and consequently $C_{g^{(n)}}\leq M^2 C_g$ by sectoriality. It follows that $g^{(n)}\to g$ in $L^2(\Omega;L^2(0,T;\gamma(H,X)))$. Therefore, using any of the known maximal tail or $L^p$-estimates (e.g.,  Proposition \ref{prop:contractionS}) we infer that the corresponding stochastic convolutions satisfy $u^{(n)}\to u$ in $L^0(\Omega;C([0,T];X))$. This implies the tail estimate in the general case.
\end{proof}

\begin{remark}[$L^p$-bounds]\label{rem:LpcaseIto}  A variant of the It\^o inequality of Theorem \ref{thm-Ito} can be proven for $\|x\|^p$ with $p\geq 2$.
Then, in the same way as \cite[Theorem 1.2]{NeeZhu} (due to the time dependence in $A$ some modifications are required in the approximation argument which are similar to the ones in the proof of Theorem \ref{thm:ExpestIto}), it is possible to recover the conclusion of Theorem \ref{thm:contractionS-new}. Tracking the constant $C_{p,D}$, this proof does not seem not to give the correct order $O(\sqrt{p})$ as $p\to \infty$, however.
\end{remark}

\begin{remark}[Random evolution families]\label{rem:Aadapted} We now indicate how Theorem \ref{thm:ExpestIto} and the result pointed at in Remark \ref{rem:LpcaseIto} can be generalised to {\em random} evolution families.
To make this notion precise we assume that for all $\om\in\Om$ a family
$(A(t,\omega))_{t\in [0,T]}$ of closed operators on $X$ is given, as well as
a $C_0$-evolution
family $(S(t,s,\om)_{0\le s\le t\le T}$
satisfying Hypothesis \ref{hyp:evol}, with estimates uniform in $\om\in\Omega$. We furthermore assume that for all $0\leq s\leq t\leq T$ and $x\in X$ the random variable $S(t,s,\cdot)x$ is strongly $\F_t$-measurable.
In what follows we will suppress the $\om$-dependence from our notation whenever it is convenient.

Under these assumptions it is not even clear how the problem should be stated to begin with, because the stochastic convolution integral $\int_0^t S(t,s)g_s\ud W_s$ is not well defined in general. The reason is that the random variables $S(t,s)x$ are assumed to be $\F_t$-measurable rather than $\F_s$-measurable, and therefore the integrand will not be progressively measurable in general. To overcome this problem one notes that in the $\Omega$-independent case, for sufficiently regular $g$ one has the almost sure identity
\begin{equation}\label{eq:MP} 
\begin{aligned}
u_t :=  \int_0^t S(t,s)g_s\ud W_s & = S(t,0) \int_0^t g_r \ud W_r
\\ & \phantom{ = } - \int_0^t S(t,s) A(s)\Bigl(\int_s^t g_r \ud W_r \Bigr)\ud s, \quad t\in [0,T].
\end{aligned}
\end{equation}
Following \cite{ProVer14}, in  the $\Omega$-dependent case we {\em define} the process $(u_t)_{t\in [0,T]}$ to be given by the expression on the right-hand side of \eqref{eq:MP} and refer to it as the {\em pathwise mild solution} of the problem $\!\ud u_t = A(t)u_t\ud t + g_t\ud W_t$.
This formula has the merit of avoiding adaptedness issues and $(u_t)_{t\in [0,T]}$ can be shown to be progressively measurable. Pathwise mild solutions were extensively studied in the parabolic setting in \cite{ProVer14}; under the assumptions spelled out below the general case can be treated along similar lines.
The problem of extending Theorem \ref{thm:contractionS-new} and the result mentioned in Remark \ref{rem:LpcaseIto} can now be formulated as proving suitable Gaussian tail estimates and $L^p$-bounds for the random variable $$u^\star:= \sup_{t\in [0,T]}\n u_t\n$$ with $u_t$ given by the right-hand side of \eqref{eq:MP}.

The proof of Theorem \ref{thm:ExpestIto} can be repeated as soon as a suitable analogue of \eqref{eq:Ito-strong} is available. This is the case under the following additional technical assumptions.
We assume that there exist continuous and dense embeddings of Banach spaces $Y_2\hookrightarrow Y_1\hookrightarrow Y_0 = X$ such that
pointwise on $\Omega$ the following conditions are satisfied almost surely:
\begin{enumerate}[\rm(1)]
 \item for all $0\le s\le t\le T$ we have $S(t,s)Y_1\subseteq Y_1$;
 \item for all $t\in [0,T]$ we have $Y_1\subseteq \Dom(A(t))$ and, for $i\in \{0,1\}$, we have $A(t)Y_{i+1}\subseteq Y_{i}$ boundedly and the process $A|_{Y_{i+1}}:[0,T]\times\Omega\to \calL(Y_{i+1}, Y_{i})$ is uniformly bounded and progressively measurable;
 \item for all $0\leq s\leq t\leq T$ and $ y\in Y_1$ we have  $S(t,\cdot)y \in W^{1,1}(0,t;X)$ and  $$\frac{{\rm d}}{{\rm d}s} S(t,s) y = -S(t,s) A(s) y.$$
\end{enumerate}
Examples where this holds are discussed in the references \cite{AcqTer87, Amann95, Lun95, Pazy, Ta1, Yagi10} cited earlier.

Suppose first that $g$ is an adapted finite rank step process $g$ with values in $Y_2$. The process $u$ defined by the right-hand side of  \eqref{eq:MP} then satisfies $u_t\in \Dom(A(t))$ and the process $(t,\om)\mapsto A(t,\om) u_t(\om)$ is uniformly bounded. Moreover, as in \cite[Proposition 4.2 and Theorem 4.7]{ProVer14} and \cite[Appendix]{KuhnNe20} one checks that
the analogue of \eqref{eq:Ito-strong} holds, i.e., almost surely we have
\begin{align*}
   u_t=\int_0^tA(s)u_s \ud s+\int_0^tg_s  \ud W_s, \qquad t\in [0,T].
\end{align*}
By following the proof of Theorem \ref{thm:ExpestIto} a Gaussian tail estimate for $u^\star$ is obtained
under the assumption that $g$ is an adapted finite rank step process with values in $Y_2$.

One would like to derive the general case by means of an
approximation argument, but this is not straightforward due to the $L^\infty$-norm in the expression for the variance.
Rather, we will first extend the $L^p$-estimate of Remark \ref{rem:LpcaseIto} to random evolution families first for adapted finite rank step processes $g$. In the $L^p$-case, the density argument can be carried out which, for arbitrary $g\in L^p_{\bF}(\Omega;L^2(0,T;\gamma(H,X)))$, gives a limiting process $(u_t)_{t\in [0,T]}$ satisfying the $L^p$-maximal estimate.
To identify the limiting process as a solution in a weak sense to the evolution equation at hand, we make the additional assumption that there exists a dense linear subspace $F\subseteq X^*$ satisfying $F\subseteq \Dom((A(s))^*)$ for all $s\in [0,T]$ and $\omega\in\Omega$ and that for all $x^*\in F$ the mapping $(s,\omega)\mapsto (A(s,\omega))^*x^*$ is strongly measurable and uniformly bounded. Under this assumption it is straightforward to check that the limiting process $u$ is a weak solution in the following sense: for all $x^*\in F^*$, almost surely we have
\begin{align*}
   \lb u_t, x^*\rb =\int_0^t \lb u_s, A(s)^* x^*\rb \ud s+\int_0^t x^* \circ g_s^*  \ud W_s, \qquad t\in [0,T].
\end{align*}

To prove the Gaussian tail estimate for general $g\in L_{\bF}^\infty(\Om;L^2(0,T;\gamma(H,X)))$ one checks that there exists adapted finite rank step processes $(g^{(n)})$ with values in $Y_2$ such that $\|g^{(n)}\|_{L^2(0,T;\gamma(H,X))}\leq \|g\|_{L^2(0,T;\gamma(H,X))}$ almost surely and $g^{(n)}\to g$ in $L^2(\Omega;L^2(0,T;\gamma(H,X)))$.
Then the previously mentioned $L^p$-estimate for $p=2$ implies that the corresponding solutions satisfy $u^{(n)}\to u$ in $L^2(\Omega;C([0,T];X))$. Therefore, the Gaussian tail estimate for the pair $(g, u)$ follows from the one for $(g^{(n)}, u^{(n)})$ with constant $\sigma^2 = 2D^2\n g \n_{L^\infty(\Om;L^2(0,T;\gamma(H,X)))}^2$.
\end{remark}

\section{Open questions}

\begin{question}
Let $(S(t))_{t\geq 0}$ be a $C_0$-semigroup on a $2$-smooth Banach space $X$ and let $W$ be a standard real-valued Brownian motion with values. Does the process $(\int_0^t S(t-s) g_s \ud W_s)_{t\ge 0}$ have a continuous version for every $g\in L^2(\Omega;L^2(0,T;X))$?
\end{question}

This question is open even in the case of Hilbert spaces $X$.
More generally one could pose the above question for $H$-cylindrical Brownian motions and processes
$g\in L^0(\Omega;L^2(0,T;\gamma(H,X)))$. A possible approach could be to prove that whenever a Banach space $X$ is $2$-smooth and $(S(t))_{t\ge 0}$ is a given $C_0$-semigroup on $X$, then
there exist an equivalent $2$-smooth norm with respect to which the semigroup is contractive.

\begin{question}
Which $C_0$-semigroups and $C_0$-evolution families admit an (approximate) invertible dilation?
\end{question}

As explained in Theorem \ref{thm:approxinvdilation}, one can deduce maximal inequalities for stochastic convolutions in case $C_0$-semigroups/evolution families admit an approximate invertible dilation.

\begin{question}
Is there a ``supertheory'' containing both the theories of stochastic integration in $2$-smooth spaces and UMD spaces as special cases?
\end{question}

Presumably, such a theory would be based on some decoupling principle, in which case one expects it to apply to the class of Banach spaces satisfying a suitable decoupling inequality. For more details on decoupling and its use in the theory of stochastic integration the reader is referred to \cite{cox2018decoupling} and references therein. A possible approach could be to answer the next open question:

\begin{question}\label{q4}
Does $2$-smoothness of $X$ imply the existence of a finite constant $C\ge 0$ such that for all $n\geq 1$ and all functions $f_j:\{-1,1\}^{j-1}\to X$ for $1\leq j\leq n$ one has the decoupling inequality
\[\E \Big\|\sum_{j=1}^n r_j f_j(r_1, \ldots, r_{j-1}) \Big\|^2 \leq C^2\E \Big\|\sum_{j=1}^n \widetilde{r}_j f_j(r_1, \ldots, r_{j-1}) \Big\|^2 \ ?\]
\end{question}
Here, $(r_j)_{j\geq 1}$ is a sequences of independent Rademacher random variables (a {\em Rademacher random variable} is a random variable taking the values $+1$ and $-1$ with equal probability $\tfrac12$) and $(\widetilde{r}_j)_{j\geq 1}$ is an independent copy of  $(r_j)_{j\geq 1}$. This decoupling inequality is satisfied in UMD spaces.

For the class of Banach spaces with type $2$, which contains the class of all $2$-smooth Banach spaces as a proper subclass, Question \ref{q4} has a negative answer. Indeed, James \cite{James78} constructed a non-reflexive space $X$ with type $2$. If the decoupling inequality would hold in this space, then \cite[Propositions 4.3.5 and 4.3.13]{HNVW16} imply that $X$ has martingale type $2$ and hence is reflexive, a contradiction.

\medskip

\noindent {\em Acknowledgment.} \
We thank the anonymous referees for helpful comments.

\newcommand{\etalchar}[1]{$^{#1}$}

\end{document}